\documentclass[a4paper,11pt,oneside,reqno]{amsart}
\usepackage{amsmath,amssymb,amsthm,mathtools}
\usepackage[dvipdfmx]{graphicx}
\usepackage{color}
\usepackage{hyperref}
\mathtoolsset{showonlyrefs}
\usepackage{bm}

\usepackage{subfigure}
\usepackage{verbatim}
\usepackage{wrapfig}
\usepackage{ascmac}
\usepackage{cases}
\usepackage{amsthm}
\usepackage{autobreak}

\usepackage{comment}
\usepackage {cancel}
\usepackage[normalem]{ulem}

  \makeatletter
    
    \@addtoreset{equation}{section}
  \makeatother
\newtheorem{theo}{Theorem}[section]

\newtheorem{lemma}[theo]{Lemma}

\newcounter{c}
\newcounter{d}
\newcounter{b}

\newcommand{\R}{\mathbb{R}} 

\newcommand{\e}{\varepsilon} 
\newcommand{\supp}{\mathrm{supp}} 
\newcommand{\dist}{\mathrm{dist}} 

\begin{document}
\title{Formation of stationary interfaces in the fast reaction limit}
\author{Yuki Tsukamoto}
\date{}

\address{Tokyo University of Science,
	162-8601, Tokyo, Japan}
\email{math.y.tsukamoto@gmail.com} 

\begin{abstract}
We study a two-component reaction-diffusion system in which one of the reaction terms becomes singularly large. 
Assuming that the initial data are nonnegative and mutually segregated, we prove that the solution converges 
to that of the heat equation in the nonreactive region, while it remains zero elsewhere. The analysis is based 
on explicitly constructed barrier functions and a comparison argument adapted to systems with asymmetric 
reaction terms, which together provide control near the interface. As a result, the limiting behavior exhibits a stationary 
phase interface determined by the initial support of the reactive component.
\end{abstract}
\maketitle
	\section{Introduction}
	Reaction-diffusion systems have long been used as a fundamental framework for modeling the interplay 
	between spatial transport and local interactions. In their standard form,
\[
  \partial_t u = D_1 \Delta u + f_1(u,v), \qquad
  \partial_t v = D_2 \Delta v + f_2(u,v),
\]
they arise in a wide variety of applications, including biology, chemistry, and materials science. 
The diffusion terms describe spatial dispersion, while the reaction terms model nonlinear interactions 
between the components.

In many problems, the reaction terms may act on a much faster time scale than diffusion. 
To analyze such regimes, one introduces a large parameter \( k > 0 \) to scale the reaction terms 
and studies the singular limit as \( k \to \infty \). A typical example is given by the system
\[
  \partial_t u_k = D_1\Delta u_k - k f_1(u_k,v_k), \qquad
  \partial_t v_k = D_2\Delta v_k - k f_2(u_k,v_k).
\]
This limit, known as the fast reaction limit, has been extensively studied across a range of settings. 
Prominent examples include diffusive chemical reactions, both irreversible and reversible \cite{hilhorstFastReactionLimit1996,eymard2001reaction},  
spatial segregation phenomena in competitive biological systems \cite{dancer1999spatial},  
and reactive solute transport in porous media \cite{bouillard2009fast}.

A well-studied class of such problems is given by balanced systems, in which both components react 
at comparable fast rates. This is typically modeled by assuming that the two reaction terms are proportional,
\( f_1 = \lambda f_2 \), where \(\lambda\) is a constant. Under suitable assumptions, these systems have been shown to converge, 
in the limit \( k \to \infty \), to one- or two-phase Stefan problems that describe the evolution 
of moving interfaces (see, for example, \cite{crooks2004spatial,dancer1999spatial,eymard2001reaction,hilhorst1997diffusion}).

Recent work has further advanced the understanding of balanced systems.  
Stephan established EDP-convergence for a linear reversible reaction-diffusion system based on 
its gradient flow structure \cite{stephan2021edp}.
 Perthame and Skrzeczkowski investigated systems with non-monotone reaction 
 functions and provided a limiting description based on Young measures \cite{perthame2023fast}.
 Crooks and Du studied the fast reaction limit for reaction-diffusion systems 
with nonlinear diffusion structure \cite{crooks2024fast}.

In balanced systems, the reaction terms are proportional, whereas in unbalanced systems, the reaction structure is asymmetric, as represented by \( f_1 \ne \lambda f_2 \).  
Conti et al.~\cite{conti2005asymptotic} and Hilhorst et al.~\cite{hilhorst2008relative} studied steady multi-component competition-diffusion systems  
that correspond to such unbalanced cases. 
Time-dependent unbalanced systems have been investigated by Iida et al.~\cite{Iida2017} and Hayashi~\cite{hayashi2021spatial},  
who analyzed the evolution of interfaces in various reaction regimes.
This asymmetry in the interaction can lead to fundamentally different limiting behavior as \( k \to \infty \),  
compared to the balanced case.

A representative example is the system
\begin{equation} \label{eq:unbalanced}
\left\{
\begin{aligned}
  \partial_t u_k &= D_1 \Delta u_k - k u_k^{m_1} v_k^{m_2} &&\text{in } Q_T := \Omega \times (0,T), \\
  \partial_t v_k &= D_2 \Delta v_k - k u_k^{m_3} v_k^{m_4} &&\text{in } Q_T, \\
  \partial_\nu u_k &= \partial_\nu v_k = 0 &&\text{on } \partial\Omega \times (0,T), \\
  u_k(\cdot,0) &= u_0, \quad v_k(\cdot,0) = v_0 &&\text{in } \Omega,
\end{aligned}
\right.
\end{equation}
where \( m_1, m_2, m_3, m_4 \ge 1 \). The difference in exponents reflects the non-symmetric 
interaction between \( u_k \) and \( v_k \). In such systems, the fast reaction limit can give rise to 
a variety of interface behaviors, depending sensitively on the reaction nonlinearities.

Iida et al.~\cite{Iida2017} systematically classified the asymptotic behavior of solutions to the 
system~\eqref{eq:unbalanced} under Neumann boundary conditions and non-overlapping initial data,
in the special case where \( D_1 = 1 \) and \( D_2 = 0 \); that is, only the \( u_k \)-component diffuses.
They considered four representative exponent settings:
\begin{itemize}
  \item \textbf{Case I}: \( (m_1, m_2, m_3, m_4) = (m_1, 1, 1, 1) \) with \( m_1 > 3 \): the interface vanishes instantly.
  \item \textbf{Case II}: \( (1, m_2, 1, 1) \) with \( m_2 > 1 \): the interface propagates with finite speed and is governed by a one-phase Stefan problem.
  \item \textbf{Case III}: \( (1, 1, m_3, 1) \) with \( m_3 > 1 \): the interface remains stationary.
  \item \textbf{Case IV}: \( (1, 1, 1, m_4) \) with \( 1 \le m_4 < 2 \): the interface moves with finite speed.
\end{itemize}
In their classification, the behavior for Case~I with \( m_1 \in (1, 3] \) and for Case~IV 
with \( m_4 \ge 2 \) remained open.
In~\cite{tsukamoto2025}, the interface in Case~I was shown to vanish instantaneously 
even for \( m_1 \in (2,3] \), under both Neumann and Dirichlet boundary conditions.

In this paper, we focus on Case IV with \( m_4 \ge 2 \) under Neumann boundary 
conditions and non-overlapping initial data.  
This parameter regime has not been addressed in the existing literature.
Furthermore, we also consider the effect of varying \( m_3 \) in the reaction term of the \( v_k \) equation.  
Our main result is stated below.
\begin{theo}\label{thm:main}
Let \( \Omega \subset \mathbb{R}^n \) be a bounded domain with \( C^2 \) boundary
and let \(2p>n+2\) be fixed.
For each $k>0$ consider the reaction-diffusion system
\begin{align} \label{eq:main system}
  \begin{cases}
    \partial_{t}u_{k}=\Delta u_{k}-k\,u_{k}v_{k} &\text{in }\Omega\times(0,T),\\
    \partial_{t}v_{k}=-k\,u_{k}^{m_{3}}v_{k}^{m_{4}} &\text{in }\Omega\times(0,T),\\
    \partial_{\nu}u_{k}=0 &\text{on }\partial\Omega\times(0,T),\\
    u_{k}(\cdot,0)=u_{0},\quad v_{k}(\cdot,0)=v_{0} &\text{in }\Omega.
  \end{cases}
\end{align}
Assume the initial data satisfy:
\begin{enumerate}
 \item     $v_0 \in L^{\infty}(\Omega)$, \ $v_0 \ge 0$, \ $v_0 \not\equiv 0$, $v_0$ is lower semicontinuous in $\Omega$;
 \item $u_0 \in C(\overline{\Omega})$, \ $u_0 \ge 0$, \ $u_0 \not\equiv 0$, \ $u_0 \in W^{2,p}(\Omega  \setminus \supp (v_0))$;
\item $u_0 v_0 \equiv 0$ in $\Omega$;
\item $\Omega \setminus \supp (v_0) \Subset \Omega $ and $\partial\supp (v_0)$ is of class $C^2$.
\end{enumerate}
Assume further that $m_3 \ge 1$, $m_4 \ge 1$, and that either $m_3 > 1$ or $m_4 \ge 2$.
Then there exists a function $u_{\infty} \in   C(\overline{Q_T}) \cap W^{2,1}_p((\Omega \setminus\supp (v_0)) \times (0,T))$
such that $u_k \to u_\infty$ uniformly in $Q_T$ as $k \to \infty$, and $u_\infty$ solves the heat equation
\begin{align}\label{eq:heat equation}
  \begin{cases}
    \partial_t u_\infty = \Delta u_\infty & \text{in } (\Omega \setminus\supp (v_0)) \times (0,T),\\
     u_\infty = 0 & \text{on } \partial(\supp (v_0)) \times (0,T),\\
    u_\infty(\cdot,0) = u_0 & \text{in } \Omega \setminus\supp (v_0),
  \end{cases}
\end{align}
with $u_\infty \equiv 0$ in $\supp (v_0)$.

Moreover, for any compact set $\Omega'\Subset \supp v_0$, one has $v_k(x,t)\to v_0(x)$ uniformly on $\Omega'\times[0,T]$.
\end{theo}
The condition \( m_3 > 1 \) or \( m_4 \ge 2 \) ensures sufficient reaction strength in 
the \( v_k \) equation to prevent the interface from persisting, allowing \( u_k \) to vanish rapidly 
in \( \supp(v_0) \) as \( k \to \infty \). This is essential for the emergence of the sharp interface 
separating the diffusive and non-diffusive regions.

We note that the case \( m_3 > 1 \) was also treated in~\cite{Iida2017}, 
but their analysis relied on the additional assumption that the limiting function, 
defined as a weak solution, is smooth.
This assumption was introduced 
as part of an adaptation of arguments originally developed for the balanced case, 
and can be considered somewhat artificial in this case.
Our approach builds on this framework by developing a direct method that avoids 
auxiliary assumptions on the limiting solution. 
By constructing an explicit subsolution adapted to the geometry of the support of \( v_0 \), 
we obtain sharp estimates 
for the reaction term and show uniform suppression of \( u_k \) in the reactive region. 
This leads to the convergence result stated in Theorem~\ref{thm:main}, 
with no need to assume regularity of the limit a priori.

In this study, we develop a direct analytical framework for the fast reaction limit 
based on explicitly constructed sub- and supersolutions.
In particular, we construct one-dimensional supersolutions adapted to the asymmetry and exponent 
structure of the reaction terms, and show that \(u_k\) is uniformly suppressed in the reactive region.
This construction is then extended to higher-dimensional domains and combined with 
a patching-type comparison principle to complete the proof of convergence.

The remainder of this paper is organized as follows.
Section 2 develops a comparison principle tailored to the singular structure of the system. 
We establish Theorem \ref{theo:comparison principle}, a comparison principle for weak sub- and supersolutions in the unbalanced setting,  
and prove Theorem \ref{combine functions}, a patching result that enables the extension of local supersolutions to general domains.
In Section 3, we construct explicit one-dimensional supersolutions in Lemmas \ref{lem:cosh} and \ref{lem:supersol}, 
depending on the reaction exponents \(m_3\) and \(m_4\). These are then combined in 
Theorem \ref{thm:main-one} to produce a more tractable supersolution. Using Theorem \ref{combine functions}, 
we extend this construction to the whole domain, completing the proof of 
Theorem~\ref{thm:main}.
	\section{Comparison principle and auxiliary functions}
We begin by summarizing the basic properties of the solution \( (u_k, v_k) \) to system \eqref{eq:main system}, 
as well as introducing the function spaces and notation used throughout the paper.

According to \cite{Hilhorst2001}, for each \( k > 0 \), there exists a unique solution \( (u_k, v_k) \) of \eqref{eq:main system} satisfying
\begin{align*} 
u_k &\in C([0,T]; C(\overline{\Omega})) \cap C^1((0,T]; C(\overline{\Omega})) \cap C((0,T]; W^{2,p}(\Omega)), \\
v_k &\in C^1([0,T]; L^\infty(\Omega)), \\
0 &\leq u_k \leq M_u, \quad 0 \leq v_k \leq M_v \quad \text{in } Q_T.
\end{align*}

For \( A \subset Q_T \) and \( \tau > 0 \), let
\[
  A_\tau := \{ (x,t) \in A \mid t \ge \tau \}.
\]
We define the functional spaces
\begin{align*}
\mathcal{U}(A) &:= \left\{ u \in L^\infty(A) \mid
  \partial_t u, \nabla u \in L^2(A_\tau)\ \text{for every } \tau >0 \right\}, \\
\mathcal{V}(A) &:= \left\{ v \in L^\infty(A) \mid
  \partial_t v \in L^1(A)\right\}.
\end{align*}
Let $f\in L^2(A)$.  We say that a function $u\in \mathcal{U}(A)$ satisfies
\[
\partial_tu \ge \Delta u - f(x,t)
\quad\text{in }A
\]
in the weak formulation if, for every nonnegative test function $\phi\in C^\infty_c(A)$,
\begin{align*}
\int_A\bigl(-u\partial_t\phi + \nabla u\cdot\nabla\phi + f(x,t)\phi\bigr)\, dx\, dt
\ge 0.
\end{align*}
If $A=Q_T$ and, in addition, one imposes the homogeneous Neumann condition $\partial_\nu u=0$ on $\partial\Omega$, 
the test-function class is extended to all of \(C^\infty(\overline{Q_T})\).  
In this case, a function $u \in \mathcal{U}(A)$ is said to satisfy
\[
\partial_tu \ge \Delta u - f(x,t),
\quad
\partial_\nu u=0
\quad\text{in }Q_T
\]
if for every nonnegative test function $\phi \in C^\infty(\overline{Q_T})$, the inequality
\begin{align*}
&\int_{Q_T}\bigl(-u\partial_t\phi + \nabla u\cdot\nabla\phi + f(x,t)\phi\bigr) \, dx\,dt\\
&+\int_{\Omega} u(x,T)\phi(x,T)- u(x,0)\phi(x,0)\,dx
\ge 0.
\end{align*}
holds. Replacing the inequality “\(\ge\)” by “\(\le\)” or “\(=\)” in the integral formulation gives
 the corresponding weak subsolution or solution definition, respectively. 
The same weak formulation applies to every \(v \in \mathcal{V}(A)\), 
with the sole difference that the spatial gradient term \(\nabla v \cdot \nabla \phi\) is omitted.

The following comparison principle can be established by adapting 
standard techniques developed for coupled cooperative systems (see, e.g., \cite{Boyadzhiev2025,hilhorst2000nonlinear}).
 Although the assumptions in our setting differ slightly from those in the
  referenced works, the structure of the system remains within the same
	 general framework, so the proof requires only minor modifications.
	 \begin{theo} \label{theo:comparison principle}
		Let $u, \tilde{u} \in \mathcal{U}(Q_T)$
		and $v, \tilde{v} \in \mathcal{V}(Q_T)$  be nonnegative functions. 
Assume they satisfy the inequalities
		\begin{equation} \label{eq:comparison principle}
			\begin{aligned}
				\partial_t u &\geq \Delta u - k uv, \quad
				\partial_t \tilde{u} \leq \Delta \tilde{u} - k \tilde{u} \tilde{v} &&\text{in } Q_T, \\
				\partial_t v &\leq -k u^{m_3} v^{m_4}, \quad
				\partial_t \tilde{v} \geq -k \tilde{u}^{m_3} \tilde{v}^{m_4} &&\text{in } Q_T, \\
				\partial_\nu u &= \partial_\nu \tilde{u}=0 &&\text{on } \partial \Omega \times (0,T), \\
				u(x,0) &\geq \tilde{u}(x,0), \quad v(x,0) \leq \tilde{v}(x,0) &&\text{for a.e. } x \in \Omega.
			\end{aligned}
			\end{equation}
			Then
			\begin{align*}
				u \geq \tilde{u} , v \leq \tilde{v} \quad  \text{in } Q_T.
			\end{align*}
		\end{theo}
		\begin{proof}
			Since $ u, \tilde{u} \in \mathcal{U}(Q_T) $ and $ v, \tilde{v} \in \mathcal{V}(Q_T) $, 
			all functions are uniformly bounded in $ Q_T $; that is, there exists a constant $ M > 0 $ such that
\begin{align*}
	\|u\|_{L^\infty(Q_T)} + \|\tilde{u}\|_{L^\infty(Q_T)} + \|v\|_{L^\infty(Q_T)} + \|\tilde{v}\|_{L^\infty(Q_T)} \leq M.
\end{align*}
			Define $U =  \tilde{u}-u$ and $V = v-\tilde{v} $. 
			We also set $U_+ = \max(U, 0)$ and $V_+ = \max(V, 0)$.
			We will show that $U_+$ and $V_+$ are identically zero.
			By \eqref{eq:comparison principle}, the following inequalities hold for any nonnegative 
			test functions $\phi, \psi \in C^{\infty}(\overline{Q_T})$:
			\begin{align} \label{eq:compari 2}
				\begin{split}
					\int_{Q_T} \left( -U  \partial_t \phi + \nabla U \cdot \nabla \phi + k ( \tilde{u} \tilde{v}-uv) \phi \right) \,  dx \,  dt
					+ \int_{\Omega} U(x,T) \phi(x,T)  \, dx&\leq 0, \\
					\int_{Q_T} \left( -V  \partial_t \psi + k (u^{m_3} v^{m_4}-\tilde{u}^{m_3} \tilde{v}^{m_4} ) \psi \right)  \,dx \,  dt 
					+ \int_{\Omega} V(x,T)  \psi(x,T)   \, dx&\leq 0.
				\end{split}
			\end{align}
			We now choose $\phi = U_+ e^{-Lt}$ and $\psi = V_+ e^{-Lt}$, where $L > 0$ is a sufficiently large constant to be determined later.
			Since they are not smooth, we approximate them by smooth functions and pass to the limit.
			Then \eqref{eq:compari 2} yields the following inequalities:
			\begin{align} \label{eq:compari 3}
				\begin{split}
					0 \geq& \int_{Q_T} e^{-Lt} \left( LU_+^2 -\frac{1}{2}\partial_t (U_+^2) + |\nabla U_+|^2  + k ( \tilde{u} \tilde{v}-uv) U_+ \right)  \,dx \,dt
					\\ &+ \int_{\Omega} U_+^2(x,T) e^{-LT}  \, dx \\
					\geq&\int_{Q_T} e^{-Lt} \left( \frac{LU_+^2}{2}   + k ( \tilde{u} \tilde{v}-uv) U_+  \right)  \,dx \,dt
					+ \int_{\Omega} \frac{U_+^2 e^{-LT}}{2}  \, dx, \\
					0  \geq& \int_{Q_T} e^{-Lt} \left( \frac{LV_+^2}{2} + k (u^{m_3} v^{m_4}-\tilde{u}^{m_3} \tilde{v}^{m_4} ) V_+  \right)  \,dx \,dt 
					+ \int_{\Omega} \frac{V_+^2 e^{-LT}}{2}  \, dx.
				\end{split}
			\end{align}
			We now estimate the terms arising from the reaction coupling. For the first term, applying Young's inequality,
			 we obtain
			\begin{align}\label{eq:compari 4}
			U_+(\tilde{u} \tilde{v} - uv) 
			&= U_+(\tilde{v}(\tilde{u} - u) + u(\tilde{v} - v) )\\
			&= \tilde{v} U_+^2 - u U_+V \\
			&\geq  - u U_+V_+ \\
			&\geq - \frac{M}{2} (U_+^2 + V_+^2).
			\end{align}
			For the second term, using Young's inequality and  mean-value theorem, we have
			\begin{align}\label{eq:compari 5}
				\begin{split}
									V_+(u^{m_3} v^{m_4}-\tilde{u}^{m_3} \tilde{v}^{m_4})
				&= V_+(u^{m_3} (v^{m_4} - \tilde{v}^{m_4}) + \tilde{v}^{m_4} (u^{m_3} - \tilde{u}^{m_3})) \\
				&\geq - \tilde{v}^{m_4}m_3 M^{m_3-1} U_+ V_+\\
				&\geq - \frac{m_3 M^{m_3+m_4-1}}{2} (U_+^2 + V_+^2) .
				\end{split}
				\end{align}
			Combining \eqref{eq:compari 3}, \eqref{eq:compari 4}, and \eqref{eq:compari 5}, we obtain
			\begin{align*}
				0\geq& \int_{Q_T}  \frac{Le^{-Lt}}{2}(U_+^2+V_+^2) - k e^{-Lt}\left(\frac{M}{2}+
				\frac{m_3 M^{m_3+m_4-1}}{2}\right)(U_+^2+V_+^2)   \,dx\, dt.
			\end{align*}
			By setting $L=2k(M+m_3 M^{m_3+m_4-1})$, we obtain 
			\begin{align}
				0\geq& \int_{Q_T}  \frac{Le^{-Lt}}{4}(U_+^2+V_+^2)  \,dx \,dt.
			\end{align}
			This implies that $U_+ = 0$ and $V_+ = 0$ in $Q_T$.
			Thus, we have $u \geq \tilde{u}$ and $v \leq \tilde{v}$ in $Q_T$.
		\end{proof}
		In order to develop the comparison principle on a finite covering, 
		we first prove the following preliminary result for two solution pairs on a common subdomain.
\begin{lemma} \label{lem:2}
		Let $k>0$, $m_3\ge 1$, and $m_4\ge 1$ be constants. 
		Let $Q_T=\Omega\times(0,T)$, where $\Omega$ is a bounded domain in $\R^n$ with Lipschitz boundary.
	Let $A \subset Q_T$ be a subdomain with Lipschitz boundary. Suppose
	\begin{align*}
		u_i\in \mathcal{U}(A), 
\quad
v_i\in \mathcal{V}(A),
\qquad
i=1,2,
	\end{align*}
	be nonnegative functions,
	and that in $A$ each pair $(u_i,v_i)$ satisfies
\begin{align*}
	\partial_t u_i &\ge \Delta u_i - ku_iv_i,\\
\partial_t v_i &\le -\,k\,u_i^{m_3}v_i^{m_4}.
	\end{align*}
	Define the index map
			\begin{align*}
				\ell(x,t):=\begin{cases}
					1, & u_1(x,t)\le u_2(x,t),\\
					2, & u_2(x,t)< u_1(x,t),
					\end{cases}
			\end{align*}
			and set $u= u_{\ell(x,t)}$ in $A$. If $v_{\ell(x,t)}$ is decreasing in $t$ for any $x\in A$,
			then there exists a function $v \in \mathcal{V}(A)$ such that
			\begin{align*}
				\partial_t u &\geq \Delta u - k uv , \\
				\partial_t v &\leq -k u^{m_3} v^{m_4} . 
			\end{align*}
\end{lemma}
\begin{proof}
	Let
	\[
	\pi_x(A)=\bigl\{x\in\Omega \mid A\cap(\{x\}\times(0,T])\neq\emptyset\bigr\},
	\]
	and define the “first‐entry” time
	\[
	t_0\colon \pi_x(A)\to[0,T), 
	\quad
	t_0(x)=\inf\bigl\{t\in(0,T]\mid(x,t)\in A\bigr\}.
	\]
	Since \(v_{\ell(x,t)}\) is nonincreasing in \(t\) and uniformly bounded on \(A\), the limit
	\[
	v_0(x)=\lim_{s\downarrow t_0(x)}v_{\ell(x,s)}(x,s)
	\]
	exists and defines \(v_0\in L^\infty(\pi_x(A))\). 
	We next introduce the notation
	\begin{align*}
		I(x,t)
=\int_{A\cap(\{x\}\times(t_0(x),t])}
u^{m_3}(x,\tau)\,d \tau.
	\end{align*}
	We then set
	\begin{align*}
	v(x,t)=
	\begin{cases}
	v_0(x)\exp\!\bigl(-kI(x,t) \bigr), & m_4=1,\\[6pt]
	\bigl((m_4-1)kI(x,t) + v_0(x)^{1-m_4}\bigr)^{-1/(m_4-1)} & m_4>1.
	\end{cases}
	\end{align*}
	From the construction of $v$, it is immediate that
	\begin{align*}
		&v\in \mathcal{V}(A),\\
		&\partial_t v = -ku^{m_3}v^{m_4}\quad\text{in }A.
	\end{align*}
	Moreover, since $v_{\ell(x,t)}$ is nonincreasing in $t$, any switch in the index \(\ell(x,t)\) 
	can only replace \(v_{\ell(x,t)}\) by an equal or smaller value.  Therefore,
	\[
	v(x,t)\ge  v_{\ell(x,t)}(x,t)
	\quad\text{in } A.
	\]
	It remains to verify the weak form of
	$\partial_t u\ge \Delta u - kuv$.
	Next, let $\phi\in C^\infty_c(A)$ be any nonnegative 
	test function. Fix $1 >\e >0$ and define
	\begin{align*}
		\eta_\e(x,t)=
		\begin{cases}
		0 & u_1\le u_2,\\
		\frac{u_1-u_2}{\e} & u_2< u_1< u_2+\e,\\
		1 & u_1\ge u_2+\e.
		\end{cases}
	\end{align*}
We define \(\phi_1 = (1 - \eta_\e)\phi\) and \(\phi_2 = \eta_\e \phi\).  
Since \(\eta_\e\) is only Lipschitz, we approximate \(\phi_1\) and \(\phi_2\) by smooth test functions, substitute them into the weak formulations for \(u_1\) and \(u_2\), and then pass to the limit.  
Summing the resulting inequalities gives
\begin{align*}
	0 \leq& \int_A \left( -u_1 \partial_t \phi_1 + \nabla u_1 \cdot \nabla \phi_1 + k u_1 v_1 \phi_1 \right)
	  \,dx \,dt\\
	&+\int_A \left( -u_2 \partial_t \phi_2 + \nabla u_2 \cdot \nabla \phi_2 + k u_2 v_2 \phi_2 \right)
	  \,dx \,dt.
\end{align*}
Consequently,
\begin{align}\label{eq:lem1}
	\begin{split}
		0 \leq& \int_A \left( -u_1 (1-\eta_\e)\partial_t \phi-u_2 \eta_\e \partial_t \phi \right)\,dx \,dt\\
		&+\int_A \left( (1-\eta_\e) \nabla u_1 \cdot \nabla \phi + \eta_\e \nabla u_2 \cdot \nabla \phi \right) \,dx \,dt\\
		&+ \int_A \phi (u_1-u_2)  \partial_t \eta_\e \,dx \,dt
		+ \int_A \phi \nabla \left( u_2-u_1  \right) \cdot \nabla \eta_\e \,dx \,dt\\
		&+ \int_A \left( k u_1 v_1 (1-\eta_\e) + k u_2 v_2 \eta_\e \right) \phi \,dx \,dt.
	\end{split}
\end{align}
By integration by parts and the definition of $\eta_\e$, the third and fourth terms become
\begin{align}\label{eq:lem2}
	\begin{split}
		&\int_A \phi (u_1-u_2)  \partial_t \eta_\e \,dx \,dt
		+ \int_A \phi \nabla \left( u_2-u_1  \right) \cdot \nabla \eta_\e \,dx \,dt\\
	=&\int_A \frac{ \e}{2} \phi \partial_t (\eta_\e^2)-\phi \e |\nabla \eta_\e|^2 \,dx \,dt\\
	\leq&\frac{ \e}{2} \int_A \eta_\e^2 \partial_t \phi \,dx \,dt.
	\end{split}
\end{align}
Since $\eta_\e^2\le1$ and $\partial_t\phi\in L^\infty(A)$, the right–hand side tends to zero as $\e \to0$.
 Moreover, since $v\geq v_{\ell(x,t)}$, the last term \eqref{eq:lem1} satisfies
\begin{align}\label{eq:lem3}
	\begin{split}
		&\int_A \left( k u_1 v_1 (1-\eta_\e) + k u_2 v_2 \eta_\e \right) \phi \,dx \,dt\\
		\leq& \int_A k\left(  u_1  (1-\eta_\e) +  u_2  \eta_\e \right)v \phi \,dx \,dt.
	\end{split}
\end{align}
Combining \eqref{eq:lem1}, \eqref{eq:lem2}, and \eqref{eq:lem3}, we obtain
\begin{align*}
	0&\leq \liminf_{\e\to 0} \bigg[\int_A \left( -u_1 (1-\eta_\e)\partial_t \phi-u_2 \eta_\e \partial_t \phi \right)\,dx \,dt\\
	&+\int_A \left( (1-\eta_\e) \nabla u_1 \cdot \nabla \phi + \eta_\e \nabla u_2 \cdot \nabla \phi \right) \,dx \,dt\\
	&+ \int_A k\left(  u_1  (1-\eta_\e) +  u_2  \eta_\e \right)v \phi \,dx \,dt \bigg]\\
	&\leq \int_A(-u\partial_t\phi + \nabla u\cdot\nabla\phi + kuv\phi)\, dx\, dt.
\end{align*}
The theorem is proved.
\end{proof}
From the above lemma, we now extend the comparison principle inductively to an arbitrary finite covering.
		\begin{theo} \label{combine functions}
			Let $\{A_i\}_{i=1}^N \subset Q_T$ be a finite collection of subdomains with Lipschitz boundaries, and set
			\begin{align*}
				A=\bigcup_{i=1}^N A_i.
			\end{align*}
			For each $i=1, \ldots, N$, let $u_i \in \mathcal{U}(A_i)$,
			 $v_i \in \mathcal{V}(A_i)$ be nonnegative functions, and assume on $A_i$ that
			 \begin{align*}
					\partial_t u_i &\geq \Delta u_i - k u_iv_i ,\\
					\partial_t v_i &\leq -k u_i^{m_3} v_i^{m_4} .
				\end{align*}
			Define the index map
			\begin{align*}
				\ell(x,t):=\min\bigl\{i\in\{1,\dots,N\}\mid u_i(x,t)=\min_{1\leq j\leq N}u_j(x,t)\bigr\},
			\end{align*}
			and set $u(x,t):= u_{\ell(x,t)}(x,t)$ in $A$. If $v_{\ell(x,t)}$ is nonincreasing in $t$ for any $x\in A$,
			$u$ is continuous in $A$, and the following condition holds: For any compact set $B \subset\subset A$
			there exists a constant $\delta >0$ and an open neighborhood $\tilde{B} \subset A$ of 
			$B \cap \bigcup_{i=1}^N \partial A_i$ such that
			\begin{align}\label{deltabound}
				u(x,t) + \delta < \min_{i \ne \ell(x,t)} u_i(x,t) \quad  (x,t) \in \tilde{B}.
			\end{align}
			Then there exists a function $v \in \mathcal{V}(A)$ such that
			\begin{align*}
				\partial_t u &\geq \Delta u - k uv, \\
				\partial_t v &\leq -k u^{m_3} v^{m_4}. 
			\end{align*}
			\end{theo}
			\begin{proof}
				We first treat the case \(N=2\); the general case follows by applying the same argument pairwise.
				Construct $v$ exactly as in Lemma \ref{lem:2}.
				Next, let \(\tilde\phi\in C^\infty_c(A)\) be any nonnegative 
	test function.  By \eqref{deltabound}, there exist open sets
	\[
	B_1,B_2\subset A
	\]
	and \(\delta>0\) such that
	\[
	\partial A_2\cap\supp\tilde\phi\subset B_1\subset A_1,
	\quad
	\partial A_1\cap\supp\tilde\phi\subset B_2\subset A_2,
	\]
	and
	\[
	u(x,t)+\delta<\min\{u_1(x,t),u_2(x,t)\}
	\quad\text{for }(x,t)\in B_1\cup B_2.
	\]
	Continuity of \(u\) then implies \(u=u_1\) on \(B_1\) and \(u=u_2\) on \(B_2\), and hence
	\[
	u_1(x,t)+\delta<u_2(x,t)\quad\text{on }B_1,
	\qquad
	u_2(x,t)+\delta<u_1(x,t)\quad\text{on }B_2.
	\]
	Since \(\supp\tilde\phi\) is compact and contained in the union
	\[
	(A_1\cap A_2)\cup\bigl((A_1\setminus A_2)\cup B_1\bigr)\cup\bigl((A_2\setminus A_1)\cup B_2\bigr),
	\]
	we may construct a nonnegative partition of unity
	$\phi,\phi_1,\phi_2\in C^\infty_c(Q_T)$ satisfying
	\begin{align*}
	\tilde\phi &= \phi + \phi_1 + \phi_2,\\
	\supp\phi &\subset A_1\cap A_2,\\
	\supp\phi_1 &\subset (A_1\setminus A_2)\cup B_1,\\
	\supp\phi_2 &\subset (A_2\setminus A_1)\cup B_2,
	\end{align*}
	and such that $u = u_i$ on $\supp\phi_i$ for $i=1,2$.  
	It follows immediately that, using $v\ge v_{\ell(x,t)}$ and the weak formulations on each $A_i$,
	\begin{align}\label{eq:compari 6}
		\begin{split}
			&\int_{Q_T} \left( -u  \partial_t (\phi_1+\phi_2) + \nabla u \cdot \nabla (\phi_1+\phi_2) 
		+ k u v (\phi_1+\phi_2) \right)  \,dx \,dt\\
		\geq &\int_{Q_T} \left( -u_1 \partial_t \phi_1 
		+ \nabla u_1 \cdot \nabla \phi_1 + k u_1 v_1 \phi_1 \right)  \,dx \,dt\\
		 &+\int_{Q_T} \left( -u_2 \partial_t \phi_2 + \nabla u_2 
		\cdot \nabla \phi_2 + k u_2 v_2 \phi_2 \right)  \,dx \,dt\\
		\geq &0.
		\end{split}
	\end{align}
	By Lemma \ref{lem:2} on the subdomain $A_1\cap A_2$, and since the construction of $v$
	is identical to that in Lemma \ref{lem:2}, we have
	\begin{align}\label{eq:compari 7}
		\int_{Q_T}\bigl(-u\partial_t\phi + \nabla u\cdot\nabla\phi + kuv\phi\bigr)\, dx\, dt \ge0.
		\end{align}
		Therefore, by \eqref{eq:compari 6} and \eqref{eq:compari 7}, the theorem is proved.
			\end{proof}

\section{Proof of Theorem \ref{thm:main}}
In this section we first construct explicit one-dimensional barriers and then use them to prove the theorem in the general case \(n \ge 1\).
Lemmas \ref{lem:cosh} and \ref{lem:supersol}, as well as Theorem \ref{thm:main-one}, are established under the one-dimensional assumption \(n = 1\), where \(\Omega \subset \mathbb{R}\).
The extension to higher dimensions is carried out at the end of the section.

We begin with an explicit supersolution based on the hyperbolic cosine function,
 corresponding to the exponent $m_3$.
\begin{lemma}\label{lem:cosh}
Let $m\ge 1$ and $a_1>0$.  Then there exists a constant $K_{1}=K_{1}(a_1,m)>0$
such that, for every $k\ge K_{1}$, there exists a number
\[
   \tilde{x}\in\bigl(k^{-1},k^{-\frac{1}{3}}\bigr)
\]
with the following properties.  Define
\[
   U(x):=k^{-2}\cosh\bigl(\sqrt{a_1k} x\bigr).
\]
Then
\begin{align}\label{eq:basic}
   U(0)=k^{-2},\quad
   &U'(0)=0,\quad
   U''(x)=a_1k\,U(x)>\sqrt{a_1k}|U'(x)|,\\
	 &\int_{-k^{-1}}^{\tilde{x}} k\,U(s)^{m}\,\mathrm{d}s<1,\label{eq:int}\\
	 	 &  \begin{cases} \label{eq:U}
      U'(\tilde{x})\ge \log k, & m>1,\\[4pt]
      U(\tilde{x})\ge 2\,k^{-\frac{2}{3}}, & m=1.
   \end{cases}
\end{align}
\end{lemma}
\begin{proof}
The identities in \eqref{eq:basic} follow at once from direct computation.

Assume first that \(m>1\).  
Choose \(\tilde{x}>0\) so that \(U(\tilde{x}) = k^{-\frac12\sqrt{m}}\).
Because \(U\) is strictly increasing for \(x>0\),
\begin{align*}
	\tilde{x}
      &=\frac{1}{\sqrt{a_1k}}
        \cosh^{-1}\!\bigl(k^{2-\tfrac12\sqrt{m}}\bigr)\\
      &=\frac{1}{\sqrt{a_1k}}
        \log\!\Bigl(k^{2-\tfrac12\sqrt{m}}
            +\sqrt{k^{4-\sqrt{m}}-1}\Bigr)
      \le \frac{\log(2k^{2})}{\sqrt{a_1k}},
\end{align*}
and, for \(k\) large, \(k^{-1}<\tilde{x}<k^{-1/3}\).
Since \(U\) is even, this gives
\[
   \int_{-k^{-1}}^{\tilde{x}} k\,U(s)^{m}\,ds
      \le 2\int_{0}^{\tilde{x}} k\,U(s)^{m}\,ds
      \le 2\,\tilde{x}\,k\,U(\tilde{x})^{m}
      \le \frac{2\log(2k^{2})}{\sqrt{a_1}}\,
           k^{\frac12-\frac12\sqrt{m}}.
\]
Because \(\tfrac12-\tfrac12\sqrt{m}<0\), the right-hand side
tends to zero as \(k\to\infty\), so \eqref{eq:int} holds for all sufficiently large \(k\).
Moreover,
\begin{align*}
	 U'(\tilde{x})
      &=k^{-2}\sqrt{a_1k}\sinh(\sqrt{a_1k}\tilde{x})\\
      &=k^{-2}\sqrt{a_1k}\sqrt{k^{4-\sqrt{m}}-1}
      \ge \frac{\sqrt{a_1}}{2}
           k^{\frac12-\frac12\sqrt{m}}
      \ge \log k,
\end{align*}
verifying the first line of \eqref{eq:U}.

Now consider the case \(m=1\).  
Set
\[
   \tilde{x}
      :=\frac{1}{\sqrt{a_1k}}
         \cosh^{-1}(2k^{4/3})
      =\frac{1}{\sqrt{a_1k}}
         \log\!\bigl(2k^{4/3}+\sqrt{4k^{8/3}-1}\bigr),
\]
so that \(U(\tilde{x}) = 2k^{-2/3}\) and, for large \(k\),
\(k^{-1}<\tilde{x}<k^{-1/3}\).
Repeating the above estimate gives
\[
   \int_{-k^{-1}}^{\tilde{x}} k\,U(s)\,ds
      \le 2\,\tilde{x}\,k\,U(\tilde{x})
      \le \frac{2\log(4k^{2})}{\sqrt{a_1}}\,
           k^{-1/6}
      < 1.
\]

Choosing \(K_{1}=K_{1}(a_1,m)\) large enough for both cases
completes the proof.
\end{proof}
Next, we construct a supersolution corresponding to the exponent $m_4$.
\begin{lemma}\label{lem:supersol}
  Let \(2 \le m\), \(0<a_2\le 1\), and \(1\le b_2\). Then there exist a function $U \in C^2([0,k^{-\frac{1}{6}}])$
	 and a constant \(K_2=K_2(a_2,b_2)>0\) such that for each \(k\ge K_2\), 
  \begin{align}
    &U(0)=k^{-\frac{2}{3}},\quad U'(0)=0, \\
    &k^{\frac{1}{8}}U'(x)<U''(x)\le a_2kU(x) \bigl((m-1)k\int_0^xU(y) \, dy + b_2^{m-1}\bigr)^{-1/(m-1)},\\
		&U'(k^{-\frac{1}{6}})\ge \frac{a_2}{8}\log k.
  \end{align}
\end{lemma}
\begin{proof}
	We first treat the case \(m=2\). 
	Define
\[
   I(x):=\int_{0}^{x}U(y)\,dy, \ 
   U''(x)=\frac{a_2kU(x)}{kI(x)+b_2}, \ 
   U(0)=k^{-2/3},\ U'(0)=0.
\]
Standard ODE theory ensures the existence and uniqueness of a function
\(U\in C^{2}\bigl([0,k^{-1/6}]\bigr)\).
	Integrating once gives
\[
   U'(x)=a_2 \left(\log\!\bigl(kI(x)+b_2\bigr)-\log b_2 \right).
\]
	Because \(U''>0\) and \(U'(0)=0\), the functions $U'$, $U$, and $I$ are monotone increasing on $[0,k^{-1/6}]$.
  Consequently $U(x)\ge k^{-2/3}$ and $I(x)\ge k^{-2/3}x$, so
\[
   U' \bigl(k^{-1/6}\bigr)
      \ge  a_2 \left(\log \bigl(k^{1-2/3-1/6}+b_2\bigr)-\log b_2\right)
      \ge \tfrac{a_2}{8}\log k
\]
for all sufficiently large \(k\).

Next we verify that \(k^{1/8}U'(x)<U''(x)\) for every \(x\in[0,k^{-1/6}]\). 
Since \(I\), \(U\), and \(U'\) are increasing, we have \(I(x)\le k^{-1/6}U(x)\)
 and \(U(x)\le k^{-1/6}U'(x)\).
Therefore,
\[
   U''(x)=\frac{a_2kU(x)}{kI(x)+b_2}\ge \frac{a_2kU(x)}{k^{5/6}U(x)+b_2}.
\]
Since \(U(x)\ge k^{-2/3}\), it follows that \(k^{5/6}U(x)\ge k^{1/6}\).
For \(k\) sufficiently large, we also have \(k^{1/6}\ge b_2\), whence
\[
  U''(x)\ge\frac{a_2}{2}k^{1/6}.
\]

Next, from $I(x)\le k^{-1/3}U'(x)$ we deduce
\[
  U'(x)=a_2\log \left(\frac{kI(x)}{b_2}+1\right) 
        \le a_2 \log \left(\frac{k^{2/3}U'(x)}{b_2}+1\right) .
\]
Whenever $U'(x)\ge b_2 \ge 1$, we may apply the elementary bound
 $\log(\xi+1)\le\log(2\xi)$ for all $\xi\ge1$, which yields
\[
 U'\le a_2 \log \Bigl(\frac{2k^{2/3}U'}{b_2}\Bigr)=a_2 \log \Bigl(\frac{2k^{2/3}}{b_2}\Bigr)+ \log (U')
 \le a_2 \log \Bigl(\frac{2k^{2/3}}{b_2}\Bigr)+ \frac{U'}{2}.
\]
Rearranging gives
\[
  U'(x)\le 2a_2\log\Bigl(\frac{2k^{2/3}}{b_2}\Bigr).
\]
Therefore
\[
  k^{1/8}U'(x)
  \le 2a_2 k^{1/8}\log\!\Bigl(\frac{2k^{2/3}}{b_2}\Bigr)
  <\frac{a_2}{2}k^{1/6}\le U''(x),
\]
for sufficiently large $k$.
	
To extend the construction to \(m>2\), we first establish an auxiliary
inequality.  For every \(z\ge 0\) we claim
\begin{align}
	 \frac{1}{z+b_2}\le
	 \bigl((m-1)z+b_2^{m-1}\bigr)^{-1/(m-1)}.\label{lem:supersol1}
\end{align}
Indeed, define
\(
   \phi(z):=(z+b_2)^{m-1}-(m-1)z-b_2^{m-1}.
\)
Because \(b_2\ge 1\) we have
\(\phi(0)=0\) and
\(
   \phi'(z)=(m-1)\bigl((z+b_2)^{m-2}-1\bigr)\ge 0,
\)
so \(\phi(z)\ge 0\) and \eqref{lem:supersol1} follows.

With \eqref{lem:supersol1} in hand, substituting \(z=k I(x)\) shows that
the very same function \(U\) constructed for the case \(m=2\) also
fulfils the differential inequality required for every \(m\ge 2\).
Thus the lemma holds with \(K_{2}=K_{2}(a_2,b_2)\).
\end{proof}
We combine Lemmas \ref{lem:cosh} and \ref{lem:supersol} with
 a moving coordinate frame to construct a more manageable one-dimensional supersolution.

Fix a speed \(s>0\) and set \( y=x+st\). Assuming that the solution takes the form
\(u(x,t) = U(y)\), \(v(x,t) = V(y)\), the parabolic inequalities
\[
   \partial_t u \ge \partial_{xx}u - kuv,
   \qquad
   \partial_t v \le -k\,u^{m_3}v^{m_4}
\] reduce to the stationary ODE system
\begin{equation*}
   sU' \ge U'' - kUV, 
   \qquad
   sV' \le -kU^{m_3}V^{m_4}.
\end{equation*}
In the next theorem, we impose a stronger differential inequality on \(U\)
than that appearing in the ODE system derived above, in order to construct 
a supersolution to the original parabolic equations in higher dimensions.
\begin{theo}
\label{thm:main-one}
Assume that $m_3 \ge 1$, $m_4 \ge 1$, and that either $m_3 > 1$ or $m_4 \ge 2$.
Fix positive constants \( a_3 \), \( b_3 \), and \( c_3 \), and assume \( 0 < s < \frac{1}{2} \).
Then there exists a constant 
\( K_3 = K_3(s,a_3,b_3,c_3,m_3,m_4) > 0 \)
such that, for every \( k \ge K_3 \), one can find a constant  
\( \hat{y} < 2(\log k)^{-1/4} \) and functions  
\( U \in W^{1,\infty}((0,\hat{y})) \) and  
\( V \in W^{1,\infty}((0,\hat{y})) \)
satisfying
\begin{align}
	  & -a_3 |U'(y)|-U''(y)+ kU(y)V(y) \ge 0, \label{mo1}\\ 
  & sV'(y) + kU(y)^{m_3}V(y)^{m_4} \le 0,\label{mo2}\\
	& U(0)>k^{-2}, U(k^{-1})=k^{-2}, U'(k^{-1})=0,V(0)=b_3,\label{mo3}\\
	 &U(\hat{y}) >c_3.\label{mo4}
\end{align}
Moreover, \(U\) is strictly decreasing on \( (0,k^{-1})\) 
and strictly increasing on \( (k^{-1},\hat{y})\).
\end{theo}
\begin{proof}
	Let us fix \(k \ge K_3 >K_1+K_2 \), where the constant \(K_3\) will be determined later, and denote
	\( \tilde{y}_1 := \tilde{x}+k^{-1} \), where \(\tilde{x} \in (k^{-1},k^{-\frac{1}{3}})\) is given by Lemma \ref{lem:cosh}.
	We first construct functions \(U_1\) and \(V_1\) on the interval \([0, \tilde{y}_1]\). 
	Let \(\gamma > 0\) be a constant to be specified later. Define
	\begin{align*}
		U_1(y) &:= k^{-2}\cosh\bigl(\sqrt{\gamma k} (y-k^{-1})\bigr),\\
		V_1(y) &:= \begin{cases}
			b_3\exp\left(-\frac{k}{s} \int^y_0 U_1^{m_3}(z) \, dz \right)  & m_4=1,\\[6pt]
			\bigl((m_4-1)\frac{k}{s}\int^y_0 U_1^{m_3}(z) \, dz + b_3^{1-m_4}\bigr)^{-1/(m_4-1)} & m_4>1.
		\end{cases}
	\end{align*}
	It follows from Lemma \ref{lem:cosh} that \((U_1,V_1)\) satisfies \eqref{mo2} and \eqref{mo3}.
	Furthermore, by \eqref{eq:int} we obtain
		\begin{align*}
		V_1(y) > \begin{cases}
			b_3\exp\left(-\frac{1}{s} \right)  & m_4=1,\\[6pt]
			\bigl((m_4-1)\frac{1}{s} + b_3^{1-m_4}\bigr)^{-1/(m_4-1)} & m_4>1.
		\end{cases}
	\end{align*}
Let \( \gamma > 0 \) be less than one half of the right-hand side.  
Then, for sufficiently large \( k \), we have
\[
-a_3 |U_1'(y)| - U_1''(y) + k U_1(y) V_1(y)
\ge k U_1(y) \left( V_1(y) - \gamma - a_3 \sqrt{\tfrac{\gamma}{k}} \right) \ge 0.
\]
Hence, \eqref{mo1} holds.

We first consider the case \( m_3 > 1 \). 
Let \( \tilde{y}_2 \in [0, \tilde{y}_1] \) be the point where \( U_1'(\tilde{y}_2) = \frac{\log k}{2} \).  
Define \( U_2 \) and \( V_2 \) by
\begin{align*}
	& U_2(y):=
  -(\log k)^{3/4}y^{2}
  +4(\log k)^{1/2} y
  +U_1(\tilde{y}_2),\\
	&V_2(y):=0.
\end{align*}
A direct computation shows that, for sufficiently large \( k \), the pair
\((U_2,V_2)\) satisfies \eqref{mo1} and \eqref{mo2}, and \(U_2\) is strictly
increasing on \( [-k^{-1},(\log k)^{-1/4}] \).
Moreover, we have
\[
U_2\left( (\log k)^{-1/4}\right)\ge 3(\log k)^{1/4}>c_3.
\]
Define \( \tilde{U}_2(y) = U_2(y - \tilde{y}_2) \) and \( \tilde{V}_2(y) = 0 \).
We observe that \( U_1(\tilde{y}_2) = \tilde{U}_2(\tilde{y}_2) \) and 
\( U_1'(\tilde{y}_2) > \tilde{U}_2'(\tilde{y}_2) \), so the assumptions of
 Theorem \ref{combine functions} are satisfied.  
 Set \(\hat{y} := \tilde{y}_2 + (\log k)^{-1/4}\) and define
 \begin{align}
	&U(y) :=
\begin{cases}
  U_1(y), & y \in [0,\tilde{y}_2), \\
  \tilde{U}_2(y), & y \in [\tilde{y}_2,\hat{y}],
\end{cases} \notag\\
&V(y) :=
\begin{cases}
  b_3\exp\left(-\frac{k}{s} \int_0^y U(z)^{m_3} \,dz\right), & m_4 = 1,\\
  \left((m_4 - 1)\frac{k}{s} \int_0^y U(z)^{m_3} \,dz + b_3^{1 - m_4} \right)^{-\frac{1}{m_4 - 1}}, & m_4 > 1.
\end{cases} \label{eq: V}
 \end{align}
Then Theorem \ref{combine functions} implies that the pair \((U,V)\) satisfies \eqref{mo1}–\eqref{mo4}.  
Moreover, \( U \) is strictly decreasing on \( (0, k^{-1}) \) and strictly increasing on \( (k^{-1}, \hat{y}) \).

Next, we consider the case \( m_3 = 1 \) and \( m_4 \ge 2 \).
Let \(U_3\) be the function obtained by applying Lemma \ref{lem:supersol} with
\(m = m_4\), \(a_2=s^{\frac{1}{m_4 - 1}}/2\), and
\begin{align*}
	b_2 =\left( (m_4 - 1)\frac{1}{s} + b_3^{1 - m_4}\right)^{\frac{1}{m_4-1}} \ge 1.
\end{align*}
Define
\begin{align*}
	V_3(y):=&\left((m_4 - 1)\frac{k}{s} \int_0^y U_3(z)\,dz + \frac{1}{s}
\left((m_4 - 1)\frac{1}{s} + b_3^{1 - m_4}\right) \right)^{-\frac{1}{m_4 - 1}}\\
=&s^{\frac{1}{m_4 - 1}}
\left((m_4 - 1)k \int_0^y U_3(z)\,dz + 
(m_4 - 1)\frac{1}{s} + b_3^{1 - m_4} \right)^{-\frac{1}{m_4 - 1}}.
\end{align*}
By construction, the pair \((U_3,V_3)\) satisfies the inequality \eqref{mo2}. 
Moreover, Lemma \ref{lem:supersol} implies that, for all sufficiently large \(k\),
\begin{align*}
	&-a_3 |U_3'(y)| - U_3''(y) + k U_3(y) V_3(y) \\
	\ge&-(a_3 k^{-\frac{1}{8}}+1)\frac{1}{2}k U_3(y) V_3(y)+k U_3(y) V_3(y)\\
	\ge&0,
\end{align*}
so that \eqref{mo1} also holds.
By Lemma \ref{lem:cosh}, there exists a point \( \tilde{y}_3 \in [0, \tilde{y}_1] \) such that 
\begin{align*}
	&U_1(\tilde{y}_3) = k^{-2/3} = U_3(0), \quad U_1'(\tilde{y}_3) > U_3'(0) = 0,\\
	&V_1(\tilde{y}_3) > \bigl((m_4-1)\frac{1}{s} + b_3^{1-m_4}\bigr)^{-1/(m_4-1)}
	> V_3(0).
\end{align*}
Hence, we can choose \( \tilde{y}_4 \in (\tilde{y}_3, \tilde{y}_1) \) and \( \tilde{y}_5 \in (0, k^{-1}) \) such that
\[
U_1(\tilde{y}_4) = U_3(\tilde{y}_5), \quad U_1'(\tilde{y}_4) > U_3'(\tilde{y}_5), \quad V_1(\tilde{y}_4) > V_3(\tilde{y}_5).
\]
We define the shifted functions
\[
\tilde{U}_3(y) := U_3(y - \tilde{y}_4 + \tilde{y}_5), \quad \tilde{V}_3(y) := V_3(y - \tilde{y}_4 + \tilde{y}_5).
\]
Then the pair \( (U_1, V_1) \) and \( (\tilde{U}_3, \tilde{V}_3) \)
 satisfy the assumptions of Theorem~\ref{combine functions}.  

 The next step is to construct a function that will be attached to the right end of \( (\tilde{U}_3, \tilde{V}_3) \).
By Lemma \ref{lem:supersol}, there exists a point \( \tilde{y}_6 \in (\tilde{y}_4,\tilde{y}_4+k^{-1/6}) \) such that
\[
\tilde{U}'_3(\tilde{y}_6) = \frac{a_2}{16} \log k.
\]
We now define
\[
U_4(y) := U_2(y - \tilde{y}_6)-U_1(\tilde{y}_2)+\tilde{U}_3(\tilde{y}_6), \quad V_4(y) := 0.
\]
Then, as in the case \( m_3 > 1 \), the assumptions of Theorem \ref{combine functions} are satisfied.

Finally, set \(\hat{y}:=\tilde{y}_6+(\log k)^{-1/4}\) and define \(U\) by assembling
\(U_1,\tilde{U}_3, U_4\) on the consecutive intervals
\([0,\tilde{y}_4)\), \([\tilde{y}_4\), \(\tilde{y}_6)\), \([\tilde{y}_6,\hat{y}]\).
We then define \(V\) via \eqref{eq: V} with this \(U\).
Exactly the same verification as in the case \(m_3>1\) shows that
the pair \((U,V)\) satisfies \eqref{mo1}–\eqref{mo4}, and
\(U\) is strictly decreasing on \((0,k^{-1})\) and strictly increasing on
\((k^{-1},\hat{y})\). This completes the proof of the theorem.
\end{proof}
Prior to the proof of Theorem \ref{thm:main}, we prove the following lemma in the general \(n\)-dimensional setting.
Before the proof, we introduce the notation \( (\supp v_0)^c := \Omega \setminus \supp(v_0) \),
\((\supp v_0)^c_d :=\{x \in \Omega \mid \dist (x,(\supp v_0)^c) <d\} \), and
\(D_d:=(\supp v_0)^c_d \times (0,T)\).
\begin{lemma} \label{lem:3.4}
Let \( u_0, v_0, \Omega, p, n \) satisfy the assumptions of Theorem \ref{thm:main},
 and let \( u_\infty \in C(Q_T) \) be a function satisfying \eqref{eq:heat equation}.
Then there exists a constant \( d_1=d_1(v_0) > 0 \) such that, 
for every \( 0 < d < d_1 \), there exists a nonnegative function
\(\overline{u}_d \in  W^{2,1}_p(D_d) \)
satisfying the following:
\begin{itemize}
  \item \(\overline{u}_d\) satisfies 
  \begin{equation}\label{eq:lem3.4.1}
      \partial_t \overline{u}_d = \Delta \overline{u}_d \quad \text{in } D_d.
    \end{equation}
 \item Fix any \( \alpha \in (0, \min\{1,2 - \frac{n+2}{p} \}) \). Then
 there exists a constant \( C_1=C_1(n,p,T,v_0,\alpha) > 0 \) such that, 
    \begin{equation}\label{eq:lem3.4.2}
  \sup_{0<d<d_1} \left( \|\overline{u}_d\|_{W^{2,1}_p(D_d)} + \|\overline{u}_d\|_{C^{\alpha, \alpha/2}(D_d)} \right)
  \le C_1 \|u_0\|_{W^{2,p}((\supp v_0)^c)}.
  \end{equation}
  \item for every \( \e_1 > 0 \), there exists a constant \( d_2 = d_2(n,p,T,u_0,v_0,\e_1) > 0 \) such that, for all \( 0 < d < \min\{d_1, d_2\} \),
 \begin{equation}\label{eq:lem3.4.3}
  u_\infty+\e_1 \ge \overline{u}_d\ge u_\infty
  \quad \text{in } (\supp v_0)^c \times [0,T].
  \end{equation}
\end{itemize}
\end{lemma}
\begin{proof}
By a standard Sobolev extension theorem (see, e.g., \cite[Section 5.4]{evans2010}),
there exists a function \( \tilde{u}_0 \in W^{2,p}(\mathbb{R}^n) \) such that
\begin{align}\label{eq:tilde u_0}
	\tilde{u}_0 = u_0 \quad \text{in } (\supp v_0)^c, \quad 
\|\tilde{u}_0\|_{W^{2,p}(\mathbb{R}^n)} \le C_2 \|u_0\|_{W^{2,p}((\supp v_0)^c)},
\end{align}
where the constant \( C_2=C_2(n,p,v_0) > 0 \). 

Since \( \Omega \) is a bounded domain and \( \partial \supp(v_0) \) is of class \( C^2 \),
 the set \( \supp(v_0) \) has only finitely many connected components.
Therefore, for sufficiently small \( d > 0 \), the boundary \( \partial (\supp v_0)^c_d \) is of class \( C^2 \).
Hence, by the global \( W^{2,1}_p \) estimates for the heat equation with Dirichlet boundary conditions
(see \cite[Chap. IV]{ladyzhenskaya1968}), 
there exist constants \( d_1 = d_1(v_0) > 0 \) and \( C_3 = C_3(n, p,T, v_0,\alpha) > 0 \) such that the following holds.
For every \( 0 < d < d_1 \), one can construct a function \( \hat{u}_d \in W^{2,1}_p(D_d) \) satisfying
\begin{align*}
  \begin{cases}
    \partial_t \hat{u}_d = \Delta \hat{u}_d & \text{in } D_d,\\
     \hat{u}_d= \tilde{u}_0 & \text{on } (\partial (\supp v_0)^c_d)\times (0,T),\\
    \hat{u}_d(\cdot,0) = \tilde{u}_0 & \text{in } (\supp v_0)^c,
  \end{cases}\end{align*}
	and
	\begin{align*}
	 \sup_{0<d<d_1} \left( \|\hat{u}_d\|_{W^{2,1}_p(D_d)} + \|\hat{u}_d\|_{C^{\alpha,\alpha/2}(D_d)} \right)
  \le C_3 \|u_0\|_{W^{2,p}((\supp v_0)^c)}.
\end{align*}
By \eqref{eq:tilde u_0}, we have \(\tilde{u}_0 \ge -d^\alpha C_2 \|u_0\|_{W^{2,p}((\supp v_0)^c)}\) in \((\supp v_0)^c_d\).
The function \(\overline{u}_d := \hat{u}_d + d^\alpha C_2 \|u_0\|_{W^{2,p}((\supp v_0)^c)}\)
 satisfies \eqref{eq:lem3.4.1} and \eqref{eq:lem3.4.2}. Using the maximum principle, we obtain
 \(\overline{u}_d\ge 0\) in \(D_d\)
 and \(\overline{u}_d\ge u_\infty\) in \((\supp v_0)^c \times [0,T]\).
On \((\partial (\supp v_0)^c)\times (0,T) \), we estimate
\begin{align*}
	&\max_{(\partial (\supp v_0)^c)\times (0,T)} \overline{u}_d\\\le&
	\max_{\partial (\supp v_0)^c_d} (\tilde{u}_0+d^\alpha C_2 \|u_0\|_{W^{2,p}((\supp v_0)^c)}) 
	+ d^\alpha C_3 \|u_0\|_{W^{2,p}((\supp v_0)^c)}\\
	\le&d^\alpha (2C_2+C_3) \|u_0\|_{W^{2,p}((\supp v_0)^c)}.
\end{align*}
Fix \(\e_1>0\) and choose \(d_2>0\) so small that 
\[
  d_2^\alpha (2C_2 + C_3)\,\|u_0\|_{W^{2,p}((\supp v_0)^c)} < \e_1.
\]
Then, by the comparison principle, \eqref{eq:lem3.4.3} follows, and the lemma is proved.
\end{proof}
We conclude this section by proving the main result.
\begin{proof}[Proof of Theorem \ref{thm:main}]
	To complete the proof of Theorem~\ref{thm:main}, we establish uniform convergence of \( u_k \to u_\infty \) in \( Q_T \), and 
\( v_k \to v_0 \) in any compact subset of \( \supp v_0 \).
Since \( u_k \) satisfies \( \partial_t u_k = \Delta u_k \) in \( (\supp v_0)^c \times (0,T) \)
and is nonnegative in \( \Omega \times (0,T) \),
the comparison principle implies that
\[
u_k \ge u_\infty \quad \text{in } \Omega \times (0,T).
\]
Fix \( \e > 0 \) and an open set \( \Omega' \Subset (\supp v_0)^c \).  
We will show that
\begin{align}
  u_k(x,t) \le u_\infty(x,t) + \e \quad &\text{in } \Omega \times [0,T], \label{eq:main u} \\
  0 \le v_0(x) - v_k(x,t) \le \e \quad &\text{in } \Omega' \times [0,T],\label{eq:main v}
\end{align}
for all sufficiently large \( k \).

Let \(\tilde{U}_1\) be the function constructed in Lemma \ref{lem:3.4}.
Set \(\e_1 = \e/2\) and choose \(d >0\) so small that \(d< \min\{d_1, d_2\}\)
 and \(\Omega' \Subset \Omega \setminus (\supp v_0)^c_d \Subset \Omega\).
Define the functions on \(D_d\) by
\begin{align*}
	U_1(x,t) &:= \tilde{U}_1(x,t) + \frac{\e}{2}, \quad V_1(x,t)=0.
\end{align*}

Next, since \(\partial(\supp v_0)\) is of class \(C^2\), the signed distance function 
\(\rho(x) := \dist(x, \partial(\supp v_0))\) is of class \(C^2\) 
in a tubular neighborhood of the boundary (e.g., \cite[Chap. 14]{GilbargTrudinger2001}). 
Here, we define \(\rho(x)\) to be positive outside \(\supp v_0\) and negative inside. 
In particular, for sufficiently small \(d > 0\), there exists a constant \(C_1 = C_1(v_0) > 0\) such that
\[
  |\nabla \rho(x)|=1, \quad |\nabla^2 \rho(x)| < C_1 \quad \text{for all } 
	x \in \{ x \in \Omega : |\rho(x)| < d \}.
\]
Define
\[
  \underline{v}_d := \min\left\{1,\; \min_{x \in \Omega \setminus (\supp v_0)^c_{d/2}} v_0(x)\right\}.
\]
Since \(v_0\) is lower semicontinuous and strictly positive in \(\supp v_0\), we have \(\underline{v}_d > 0\).
Let \((\tilde{U}_2, \tilde{V}_2)\) be the pair of functions on \([0, \hat{y}]\)
constructed in Theorem \ref{thm:main-one}, with the following parameter choices:
\[
s := \min\left\{\frac{1}{4}, \frac{d}{8T} \right\}, \quad
a_3 := C_1+1 , \quad
b_3 := \left(\frac{\underline{v}_d}{2}\right)^{-1}, \quad
c_3 := 1 + \|u_0\|_{L^\infty(\Omega)}.
\]
Define 
\[
  Q^2 := \left\{(x,t)\in \Omega\times[0,T] \Bigm|
            0 \le \rho(x)+st+\frac{3d}{4} \le \hat{y}\right\},
\]
and set
\[
  U_2(x,t) := \tilde{U}_2 \left(\rho(x)+st+\frac{3d}{4}\right),\qquad
  V_2(x,t) := \tilde{V}_2 \left(\rho(x)+st+\frac{3d}{4}\right)
\]
for \((x,t)\in Q^2\). Then we obtain
\begin{align*}
	\partial_t U_2 - \Delta U_2+ k U_2 V_2
	=&s\tilde{U}'_2-\tilde{U}''_2 |\nabla \rho|^2-\tilde{U}'_2 \Delta \rho  + k\tilde{U}_2 \tilde{V}_2\\
	\ge & -(C_1+1)\tilde{U}'_2-\tilde{U}''_2+ k\tilde{U}_2 \tilde{V}_2 \ge 0,\\
	\partial_t V_2 + k U_2^{m_3} V_2^{m_4} 
	=&s\tilde{V}'_2 + k\tilde{U}_2^{m_3} \tilde{V}_2^{m_4} \le 0.
\end{align*}

Next, set
\[
  Q^{3}
  :=\left\{(x,t)\in\Omega\times[0,T] \middle|
        x\in\Omega\setminus\left(\supp v_{0}\right)^{c}_{st+\frac{3d}{4}-k^{-1}}
     \right\},
\]
and define, for \((x,t)\in Q^{3}\),
\[
  U_{3}(x,t):=\frac{k^{-2}+\min\{\tilde{U}_2(0),\,3k^{-2}\}}{2},
  \qquad
  V_{3}(x,t):=\underline{v}_d\,
              \exp\left(-2^{m_3+1}k^{-1}t\right).
\]
Then we have
\begin{align*}
	&\partial_t U_3 - \Delta U_3 + k U_3 V_3
	= k U_3 V_3 \ge 0,\\
	&\partial_t V_3 + k U_3^{m_3} V_3^{m_4} 
	=-2^{m_3+1}k^{-1}V_3 + 2^{m_3}k^{-1} V_3 \le 0.
\end{align*}

Next we verify the assumptions of Theorem \ref{combine functions}. 
Set \(\partial Q^{2,3}:=\partial Q^{2}\cap Q^{3}\) and
\(\partial Q^{2,1}:=\partial Q^{2}\setminus\partial Q^{2,3}\).
By construction, we have, on the corresponding boundary pieces,
\begin{align*}
	&U_{3}\le 2k^{-2}< \frac{\e}{2}\le U_{2}       &&\text{on } \partial D_d,\\
	&U_{1}<1+\|u_0\|_{L^\infty(\Omega)}\le U_{2}       &&\text{on } \partial Q^{2,1},\\
	&U_{3}<\tilde{U}_2(0)=U_{2}       &&\text{on } \partial Q^{2,3},\\
	&U_{2}=k^{-2}<U_{3} &&\text{on } \partial Q^3,
\end{align*}
and, moreover,
\[
  \min_{\Omega_{3}} V_{3}
  =\underline v_{d}
   \exp\left(-2^{m_{3}+1}k^{-1}T\right)
  \ge \tilde V_{2}\left(k^{-1}\right)\ge 0= V_{1}.
\]
Hence the assumptions of Theorem \ref{combine functions} are satisfied. Define
\begin{align*}
	U&=\min\left\{U_{1},U_{2},U_{3}\right\},\\
  V(x,t)&=
  \begin{cases}
    \underline{v}_d(x)\exp\left(-k\!\int_0^t\! U(x,\tau)\,d\tau\right) & \text{if } m_4=1,\\[2pt]
    \left((m_4{-}1)k\!\int_0^t\! U(x,\tau)\,d\tau + \underline{v}_d^{1-m_4}\right)^{\frac{1}{1-m_4}} & \text{if } m_4>1.
  \end{cases}
\end{align*}
Then we obtain
\[
  \partial_{t}U-\Delta U+kUV\ge 0,\quad
  \partial_{t}V+kU^{m_{3}}V^{m_{4}}\le 0,\quad
  \partial_{\nu}U=\partial_\nu U_3 =0,
\]
with \(U(x,0)\ge u_{0}(x)\) and \(V(x,0)\le v_{0}(x)\) in \(\Omega\).
Applying Theorem 2.1 yields \( U \ge u_k \) and \( V \le v_k \) in \( Q_T \), which proves \eqref{eq:main u}.

Finally, on \(\Omega'\), we have
\[
  \partial_{t}v_{k}
  \ge
  -kU_{3}^{m_{3}}\left(v_{0}(x)\right)^{m_{4}}
  \ge
  -2^{m_{3}}k^{-1}\,\|v_{0}\|_{L^{\infty}(\Omega)}^{m_{4}},
\]
which implies
\[
  v_k(x,t) \ge v_0(x) - 2^{m_3}k^{-1} \|v_0\|_{L^\infty(\Omega)}^{m_4} t.
\]
Hence \eqref{eq:main v} follows for sufficiently large \(k\), and the proof of Theorem \ref{thm:main} is complete.
\end{proof}


\end{document}